\documentclass[12pt]{amsart}
\usepackage{amssymb,amsmath,amsthm}
\numberwithin{equation}{section}
\begin{document}

\theoremstyle{plain}
\newtheorem{theorem}{Theorem}[section]
\newtheorem{lemma}[theorem]{Lemma}
\newtheorem{proposition}[theorem]{Proposition}
\newtheorem{corollary}[theorem]{Corollary}
\newtheorem{conjecture}[theorem]{Conjecture}

\def\mod#1{{\ifmmode\text{\rm\ (mod~$#1$)}
\else\discretionary{}{}{\hbox{ }}\rm(mod~$#1$)\fi}}

\theoremstyle{definition}
\newtheorem*{definition}{Definition}

\theoremstyle{remark}
\newtheorem*{remark}{Remark}
\newtheorem{example}{Example}[section]
\newtheorem*{remarks}{Remarks}

\newcommand{\cc}{{\mathbb C}}
\newcommand{\qq}{{\mathbb Q}}
\newcommand{\rr}{{\mathbb R}}
\newcommand{\nn}{{\mathbb N}}
\newcommand{\zz}{{\mathbb Z}}
\newcommand{\pp}{{\mathbb P}}
\newcommand{\al}{\alpha}
\newcommand{\be}{\beta}
\newcommand{\ga}{\gamma}
\newcommand{\ze}{\zeta}
\newcommand{\om}{\omega}
\newcommand{\mz}{{\mathcal Z}}
\newcommand{\mi}{{\mathcal I}}
\newcommand{\ep}{\epsilon}
\newcommand{\la}{\lambda}
\newcommand{\de}{\delta}
\newcommand{\De}{\Delta}
\newcommand{\Ga}{\Gamma}
\newcommand{\si}{\sigma}
\newcommand{\Exp}{{\rm Exp}}
\newcommand{\legen}[2]{\genfrac{(}{)}{}{}{#1}{#2}}
\def\End{{\rm End}}

\title{On the sums of two cubes}

\author{Bruce Reznick}
\address{Department of Mathematics, University of 
Illinois at Urbana-Champaign, Urbana, IL 61801} 
\email{reznick@math.uiuc.edu}

\author{Jeremy Rouse}
\address{Department of Mathematics, Wake Forest University,
  Winston-Salem, NC 27109}  
\email{rouseja@wfu.edu}
\thanks{The second author was supported by NSF grant DMS-0901090}
\begin{abstract}
We solve the equation $f(x,y)^3 + g(x,y)^3 = x^3 + y^3$ for 
homogeneous $f, g \in \mathbb C(x,y)$, completing an  investigation begun
by Vi\`ete in 1591. The usual addition law for elliptic curves and composition
give rise to two binary operations on the set of solutions. We show
that a particular subset of the set of solutions is ring-isomorphic to
$\zz[e^{2 \pi i / 3}]$.
\end{abstract}
\date{\today}
\maketitle
\section{Introduction}
In 1591, Fran\c cois Vi\`ete published a revolutionary work on 
algebra
 which has been translated into English \cite{V} as {\it The Analytic
   Art}.  Vi\`ete's 
``Zetetic XVIII'' \cite[p.145] {V} is: 

\begin{quote}
Given two cubes, to find
  numerically two other 
cubes the sum of which is equal to the difference between those 
that are given.

\medskip

 \noindent
Let the two given cubes be $B^3$ and $D^3$, the first the 
greater, the second the smaller. Two other cubes are to be 
found, the sum of which is equal to $B^3-D^3$. Let $B-A$ be the 
root of the first one that is to be found, and let $B^2A/D^2 - D$ 
be the root of the second. Forming the cubes and comparing them 
with $B^3-D^3$, it will be found that $3D^3B/(B^3+D^3)$ equals 
$A$. The root of the first cube to be found, therefore, is 
$[B(B^3-2D^3)]/(B^3+D^3)$ and of the second is  
$[D(2B^3-D^3)]/(B^3+D^3)$. And the sum of the cubes of these 
is equal to $B^3-D^3$ .... [So it is if] $B$ is 2 and $D$ 1: The 
cube of the root 6 will equal the individual cubes of 3, 4 and 5. 
When, therefore, the cubes of $6x$ and $3x$ are given, the cubes 
of $4x$ and $5x$ will appear and the sum of the latter will be 
equal to the difference between the former.
\end{quote}

Vi\`ete worked at the dawn of algebra, when mathematicians were 
not yet comfortable with negative numbers;  his work can be put 
into somewhat more modern terminology by setting $B = x$ and $D = 
-y$. Vi\`ete's formula then becomes:

\begin{equation}
\label{viete} 
x^3 + y^3 = \left(\frac {x(x^3+2y^3)}{x^3-y^3} 
\right)^3 + \left(\frac
  {y(y^3+2x^3)}{y^3-x^3} \right)^3. 
\end{equation}

Equation (1.1) is well-known in number theory; its iteration 
shows that any sum of two cubes over $\qq$ (except those of the 
form $d^{3}$ and $2d^{3}$) has infinitely many such 
representations. See for example \cite[\S 13.7, 21.11]{HW}. 
Continuing Vi\`ete's example,
\begin{equation}
\label{iter} 
189 = 6^3 + (-3)^3 = 4^3 + 5^3 = \left(-\tfrac{1256}{61} \right)^3 + 
\left(\tfrac {1265}{61}\right)^3 =  \cdots
\end{equation}

In this paper, we find all solutions to
\begin{equation}
\label{E:mainequation}
  f^{3}(x,y) + g^{3}(x,y) = x^{3} + y^{3},
\end{equation}
where $f(x,y)$ and $g(x,y)$ are  homogeneous rational 
functions over $\cc$. Upon finding a common denominator for $(f,g)$, the
equation in \eqref{E:mainequation} becomes
\begin{equation}
\label{titleequation}
 p^3(x,y) + q^3(x,y) = (x^3+y^3)r^3(x,y), 
\end{equation} 
where $p,q,r \in \cc[x,y]$ are homogeneous polynomials ({\it forms}),
$f = p/r$ and $g = q/r$. The degree of the solution is defined to be
$\deg(p) = \deg(q) = 1 + \deg(r)$.

In projective terms,
\begin{equation}\label{homog}
(f : g : 1) = (p : q : r).
\end{equation}
Our principal definition is the following: let
\begin{equation}\label{E:def}
{\mathcal V}= \{v = (p : q : r) : \text{where $p,q,r \in \cc[x,y]$ are
  forms and satisfy } \eqref{titleequation}\}.
\end{equation}
A solution to \eqref{E:def} with $r \ne 0$ is projectively equivalent
to $(p/r : q/r : 1)$ and we will denote solutions of this type by
$(p/r,q/r)$ or $(f,g)$. However, there are three solutions to 
\eqref{titleequation} ``at
infinity'' with $r = 0$, namely $(1 : -1 : 0)$, $(1 : -\omega : 0)$,
and $(1 : -\omega^{2} : 0)$, where
\begin{equation}
\om := \Exp\left(\tfrac{2\pi i}3\right) = -\tfrac 12 + \tfrac{\sqrt 3}2 i.
\end{equation}

We observe that if $\pi$ is irreducible and $\pi | p,q$ in
\eqref{titleequation}, then $\pi^3 | (x^3+y^3)r^3$, hence, 
$\pi^2 | r^3$ (at least), so $\pi | r$. Similarly, if $\pi | p,r$,
then $\pi | q$ 
and if $\pi | q,r$, then $\pi | p$. Since $r$ is a common denominator,
no two of $\{p,q,r\}$ have a common factor. 

One would ordinarily say that, if $x^3 + y^3 = f_1^3 
+ g_1^3 = f_2^3 + g_2^3$ and $\{f_1^3,g_1^3\} =  
\{f_2^3,g_2^3\}$, then $(f_1,g_1)$ and $(f_2,g_2)$ are the same 
solution;  however as ``points'' on \eqref{E:mainequation}, each  
solution occurs 18-fold : as $(\om^j f, \om^k g)$ and $(\om^k g, \om^j 
f)$, where $j,k \in \{0,1,2\}$. We shall call these elements of
$\mathcal V$ the {\it affiliates} of $(f,g)$.

We now list $v = (p : q : r) \in \mathcal V$ (up to affiliation) with degree 
$\le 12$, with the convention that subscripts given below to $p,q,r$ 
will be inherited by $f = p/r, g = q/r$, and $v$. 
(That these are the only such elements will follow from Theorem~\ref{mainthm}.)

Let
\begin{equation}
\ze := \ze_{12} = \Exp\left(\tfrac {\pi i}6\right) = \tfrac{\sqrt 3} 2
+ \tfrac i2, 
\end{equation}
so that $\ze + \ze^{-1} = \sqrt 3$ and $\ze^3 + \ze^{-3}=0$. 
We note that there are two solutions of degree $7$, the second of
which is the complex conjugate of $v_7$ in the table below. 

\begin{tabular}{c|lll}
Degree & Solution\\
\hline
$1$ & $p_{1} = x$\\ 
    & $q_{1} = y$\\ 
    & $r_{1} = 1$\\
$3$ & $p_{3} = \zeta^{-1} x^{3} + \zeta y^{3}$\\ 
    & $q_{3} = \zeta x^{3} + \zeta^{-1} y^{3}$\\
    & $r_{3} = \sqrt{3} xy$\\
$4$ & $p_{4} = x(x^{3} + 2y^{3})$\\ 
    & $q_{4} = -y(y^{3} + 2x^{3})$\\ 
    & $r_{4} = x^{3} - y^{3}$\\
$7$ & $p_{7} = x \left(x^{6} + (1 + 3 \omega) (x^{3} y^{3} + y^{6})\right)$\\ 
    & $q_{7} = y\left( (1 + 3 \omega)(x^{6} + x^{3} y^{3}) + y^{6}\right)$\\ 
    & $r_{7} = x^{6} + (1 - 3 \omega) x^{3} y^{3} + y^{6}$\\
$9$ & $p_{9} = -x^{9} + 3x^{6} y^{3} + 6x^{3} y^{6} + y^{9}$\\ 
    & $q_{9} = x^{9} + 6x^{6} y^{3} + 3x^{3} y^{6} - y^{9}$\\
    & $r_{9} = 3xy (x^{6} + x^{3} y^{3} + y^{6})$\\
$12$ & $p_{12} = -3(x^{3} - y^{3})^{3} (x^{3} + y^{3})
- (1 + 2 \omega)(x^{3} (x^{3} + 2y^{3})^{3} + y^{3} (y^{3} + 2x^{3})^{3})$\\
    & $q_{12} =-3(x^{3} - y^{3})^{3} (x^{3} + y^{3})
- (1 + 2 \omega^2)(x^{3} (x^{3} + 2y^{3})^{3} + y^{3} (y^{3} + 2x^{3})^{3})$\\ 
    & $r_{12} = 6xy (x^{3} - y^{3})(2x^{3} + y^{3}) (x^{3} + 2y^{3})$\\
\end{tabular}

Observe that
\begin{equation}
\frac{\ze}{\sqrt {3}} = \frac{2+\om}3, \qquad 
\frac{\ze^{-1}}{\sqrt {3}} = \frac{2+\om^2}3
\end{equation}
so that $f_3,g_3 \in \qq[\om](x,y)$; see
Theorem \ref{conseqthm}(3) below. By setting $x = 6$ and $y =
-3$ in $v_{9}$, we obtain the rational solution $189 =
\left(\frac{219}{38}\right)^{3} 
+ \left(\frac{-51}{38}\right)^{3}$ to \eqref{iter}.

We remark that $v_4$ appears in \cite[pp.550-1]{Di}. In 1877, Lucas
(see \cite[p.574]{Di}) proved a formula equivalent to $v_9$; namely,
if $x^3 + y^3 = A z^3$, then $q_9^3(x,y) + p_9^3(x,y) =
A z^3r_9^3(x,y)$. In 1878, Lucas (see \cite[p.575]{Di}) gave the
identity
\begin{equation}\label{E:lucas}
\begin{gathered}
(-x^3+3x^2y+6xy^2+ y^3)^3 + (x^3+6x^2y + 3xy^2 - y^3)^3 =\\
3^3xy(x+y)(x^2+xy+y^2)^3, 
\end{gathered}
\end{equation}
which Desboves later observed (see \cite[p.575]{Di}) is equivalent to Lucas'
previous identity upon taking $(x,y) \mapsto (x^3,y^3)$. 

Even though $v_3$ is not  real,
its components become real under any map $(x,y) \mapsto (\al x + \be
y, \bar \al x+ \bar \be y)$.  
This map is invertible provided $\al\bar \be$ is not real.  Taking  $(\al,\be) =
(\ze^{-1},\ze^{1})$ in the expression $f_3^3(x,y) + g_3^3(x,y) =
(x^3+y^3)r_3^3(x,y)$ yields \eqref{E:lucas}. Thus, it can be argued
that the first ``new'' solution to \eqref{E:mainequation} in the table
is $v_7$.  The previous interest in \eqref{E:mainequation} required
solutions over $\mathbb Q$. We shall show in Theorem
\ref{conseqthm}(7) that rational solutions
 only occur for square degree. Since the
solution of degree 16 arises from iterating $v_4$, our first truly new
solution over $\mathbb Q$ has degree 25.

The set $\mathcal V$ is invariant under a large number of symmetries,
and the examples given above show that $v \in \mathcal V$ itself may  be
symmetrical.  For example, if 
$v(x,y) \in \mathcal V$, then $v(y,x)$, $\overline{v(x,y)}$ (the 
complex conjugate of $v(x,y)$), $v(x, \om y)$ and all 
combinations thereof are also in  $\mathcal V$. Further, if $v 
= (f,g), v'=(f',g') \in \mathcal V$, then there is a 
natural composition, implicit 
already in the iterations of \eqref{iter}. To be specific,  we 
define $w = (h,k) = v\circ v'$, by
\begin{equation}\label{E:circ}
h(x,y) = f(f'(x,y),g'(x,y)), \quad k(x,y) = g(f'(x,y),g'(x,y))
\end{equation}
It follows from
\begin{equation}
\begin{gathered}
h^3(x,y) + k^3(x,y) =  f^3(f'(x,y),g'(x,y)) +
g^3(f'(x,y),g'(x,y))\\
 = (f'(x,y))^3 + (g'(x,y))^3 = x^3 + y^3
\end{gathered}
\end{equation}
that $v \circ v' \in \mathcal V$ as well. The connections among $v_3$,
\eqref{E:lucas} and $v_9$ are equivalent to the equation  $v_3 \circ
v_3 = v_9$; it turns out that  $v_3 \circ v_4 = v_4 \circ v_3 =
v_{12}$.  The homogeneous version of composition
(which applies to infinite solutions as well) is given as follows.
If $v = (p : q : r)$ and $v' = (p' : q' : r')$, then
\begin{equation}\label{E:circforms}
\begin{gathered}
  v \circ v' = 
  (p(p'(x,y),q'(x,y)) : q(p'(x,y), q'(x,y)) : r'(x,y) r(p'(x,y), q'(x,y))).
\end{gathered}
\end{equation}
Thus,  $d(v\circ v') = d(v)d(v')$, unless there are common
factors in \eqref{E:circforms}. It can be shown directly that this
cannot happen, but it will also follow from our main work.

The first principal result of this paper is the following theorem.
\begin{theorem}
\label{mainthm} Under the usual addition on elliptic curves (see 
Section~\ref{pointadd} for more details), $\mathcal V$ is an 
abelian group isomorphic to $\zz + \zz + \zz/3\zz$. The 
generators of infinite order may be taken to be $h_1 = (x,y)$ and 
$h_2 = (\om x, \om y)$; the element of order 3 is $h_0 = (1: 
-\om:0)$, a solution of \eqref{titleequation} ``at infinity''.  
Further, $d(m h_1 + n h_2 + t h_0) = m^2 - mn + n^2$.

Moreover, the subgroup $\mathcal V_1 = \{ m h_1 + n h_2\}$ is
ring-isomorphic to $\zz[\om]$ under the identification $R(m
h_1 + n h_2) = m + n \om$, with addition on curves and
composition in $\mathcal V_1$ corresponding to addition and
multiplication in $\zz[\om]$.
\end{theorem}

This result has myriad consequences on the nature of
the solutions $(f(x,y),g(x,y))$; these are collected as our other main
theorem. 

\begin{theorem}
\label{conseqthm}
\ 
\begin{enumerate}
\item If $v \in \mathcal V$, then $(g(x,y))^3 = (f(y,x))^3$.

\item If $v, v' \in \mathcal V$, then $v \circ v'$ and $v' \circ v$
  are affiliates.

\item If $v \in \mathcal V$, then $f, g \in \qq[\om](x,y)$.

\item If $d(v) = 3d'$, then some affiliate of $v$ can be written as
$\tilde v \circ v_3$, where $d(\tilde  v) = d'$. Hence
 there exist forms $P,Q,R$ so that
$$
p(x,y) = P(x^3,y^3), \quad q(x,y) = Q(x^3,y^3), \quad
  r(x,y) = xyR(x^3,y^3)
$$ 

\item If $d(v) = 3d'+1$, then, up to a permutation of $(p,q)$, there
  exist forms $P,Q,R$ so that  
$$
p(x,y)
  = xP(x^3,y^3), \quad q(x,y) = yQ(x^3,y^3), \quad r(x,y) = R(x^3,y^3)
$$

\item In no case is $d(v) \equiv 2 \mod 3$
and no monomial appearing in any $p,q,r$ has an exponent congruent to
2 mod 3. 

\item The real solutions are precisely those of the form
$v = m h_{1}$ for $m \in \zz$ and have $d(v) = m^{2}$. 
The solutions of the form $v = (f, \bar f)$ are precisely those of the form
$mv_{3}$ and have $d(v) = 3m^{2}$.

\item If $f(d)$ denotes the number of solutions with
  degree $d$ (counting each collection of affiliates as one solution), then
\begin{equation}\label{E:refsugg}
  f(d) = \sum_{e | d} \legen{e}{3},
\end{equation}
where $\legen{\cdot}{3}$ denotes the usual Legendre symbol. Thus,
the degree of any solution has the form $m^{2} \prod_{j} p_{j}$ where
the $p_{j}$ are primes $\equiv 0, 1 \pmod{3}$.
\end{enumerate}
\end{theorem}

Here is the organization of the paper. In Section~\ref{pointadd}, we
review the addition law for points on elliptic curves. This endows
$\mathcal{V}$ with the structure of an abelian group. We then analyze
a subgroup $\mathcal{V}_0$ of $\mathcal V$ and prove that
Theorem~\ref{conseqthm} is true for $\mathcal{V}_{0}$. We define a
ring isomorphism between a
particular subset $\mathcal{V}_{1}$ of $\mathcal{V}$ (under
addition and composition c.f. \eqref{E:circ}) and the ring $\zz[\om]$.
 In
Section~\ref{mwtheory}, we prove that another subset of $\mathcal V$,
$\mathcal{V}_{\infty}$,   is isomorphic to the 
endomorphism ring of the elliptic curve \eqref{titleequation} and use
this to prove that $\mathcal{V}_0 = \mathcal V$ and  $\mathcal{V}_1 =
\mathcal V_{\infty}$.  In
Section~\ref{more}, we discuss the implication of these results for
a few related Diophantine equations.

We remark that many (but not all) of these results can also be derived in
an entirely elementary way. We shall present this approach in
\cite{RR}.

We also happily acknowledge helpful conversations with Bruce Berndt
and Ken Ono, and we wish to thank the anonymous referee for helpful comments
that improved the exposition.

\section{Point addition}
\label{pointadd}

A general reference for this section is \cite{Silverman}. The first part
of the presentation has been heavily influenced by \cite{S}, where addition is
discussed on the curve $X^3 + Y^3 = A$. It is implicit in \cite{S}
that $A \in \mathbb C$, although this is formally unnecessary.

Addition is defined on elliptic curves using a few basic rules. If three points
$P,Q,R$ lie on a line, then  $P+Q+R = 0$. This operation
can be shown to be associative and the set of points on the curve
forms an abelian group; see, e.g. \cite[p.62]{Silverman}.  
This is even true if we look at ``curves'' whose coordinates are rational
functions. 

 We consider the curve
\begin{equation}
C: X^3 + Y^3 = A,
\end{equation}
where for the moment we will be vague about the underlying space for
$(X,Y)$ and the nature of $A\neq 0$. The additive inverse is given by 
\begin{equation}
\label{sumsym}
(X,Y) + (Y,X) = 0,
\end{equation}
where $0$ is the additive identity, to be identified below as a point
at infinity on the curve.
To find the explicit value of the sum, we parameterize the line
through two points on $C$. If 
$(X_1,Y_1), (X_2,Y_2) \in C$, $(X_1,Y_1) \neq (X_2,Y_2)$,
then the condition $\la(X_1,Y_1)
  + (1-\la)(X_2,Y_2) \in C$ implies that
\begin{equation}
\begin{gathered}
(\la X_1 + (1-\la)X_2)^3 +  (\la Y_1 + (1-\la)Y_2)^3 = A \\
 =
(\la^3 + (1-\la)^3)A + 3\la(1-\la)A \implies \\
\la(1-\la)\left(\la (X_1^2X_2 + Y_1^2Y_2) + (1-\la) (X_1X_2^2 +
  Y_1Y_2^2) -A\right) = 0\\ 
\implies \la = 0, \la =1 \text{ or }\la = \frac {A - (X_1X_2^2 +
  Y_1Y_2^2)}{(X_1^2X_2 +  Y_1^2Y_2)- (X_1X_2^2 + Y_1Y_2^2)},   
\end{gathered}
\end{equation}
provided $X_1^2X_2 +  Y_1^2Y_2 \neq X_1X_2^2 + Y_1Y_2^2 $. Ignoring
$\la = 0,1$, and assuming this condition holds, we have  
\begin{equation}
(X_1,Y_1) + (X_2,Y_2) + (W,Z) = 0,  
\end{equation}
where
\begin{equation}
\label{sum1}
\begin{gathered}
W = \frac{A(X_1-X_2) +
  Y_1Y_2(X_2Y_1-X_1Y_2)}{(X_1^2X_2 + 
  Y_1^2Y_2)- (X_1X_2^2 + Y_1Y_2^2)} \\
Z = \frac{A(Y_1-Y_2) +
  X_1X_2(X_1Y_2-X_2Y_1)}{(X_1^2X_2 + 
  Y_1^2Y_2)- (X_1X_2^2 + Y_1Y_2^2)} \\
\end{gathered}
\end{equation}
and so
\begin{equation}
\label{sum2}
(X_1,Y_1) + (X_2,Y_2) = (Z,W),
\end{equation}
again provided that the denominator in \eqref{sum1} is not zero.

If $X_1^2X_2 +  Y_1^2Y_2 = X_1X_2^2 + Y_1Y_2^2 = B$, say, then
$(X_1-X_2)^3 + (Y_1-Y_2)^3 = A - 3B + 3B - A = 0$, hence  $Y_1 
 - Y_2 = -\om^j(X_1-X_2)$ for $j =$ 0, 1 or 2. If $X_1 = X_2$,
 then $Y_1 = Y_2$. Otherwise, $X_1 - X_2 \neq 0$ and
\begin{equation}
\begin{gathered}
0 = X_1^2X_2 +  Y_1^2Y_2 -(X_1X_2^2 + Y_1Y_2^2) = X_1X_2(X_1-X_2) +
Y_1Y_2(Y_1-Y_2)\\ = (X_1-X_2)(X_1X_2 - \om^j Y_1Y_2) \implies X_1X_2 -
\om^jY_1(Y_1 + \om^j(X_1-X_2)) = 0 \\ \implies -\om^j(Y_1 + \om^j X_1)(Y_1
- \om^j X_2) = 0
\end{gathered}
\end{equation}
If $Y_1 = - \om^j X_1$, then $A = X_1^3 + Y_1^3 =0$. 
Otherwise, $Y_1 = \om^j X_2$ and so 
$Y_2 =  \om^j X_1$; thus, $(X_2,Y_2) = (\om^{2j}Y_1, \om^j X_1)$. To
summarize: the three instances in which we cannot add distinct points
according to \eqref{sum1} and \eqref{sum2} are
\begin{equation}
(X_1,Y_1) + (Y_1,X_1), \quad (X_1,Y_1) + 
(\om  Y_1,\om^2 X_1), \quad  (X_1,Y_1) + (\om^2  Y_1,\om X_1).
\end{equation} 
Observe that, if $(X_2,Y_2) = (\om^{2j}Y_1, \om^j X_1)$, then the
numerators of $(W,Z)$ are, for $j=0,1,2$, 
\begin{equation}
\begin{gathered}
A(Y_1 - \om^jX_1)+ \om^{2j}X_1Y_1(\om^j X_1^2 -\om^{2j} Y_1^2) \\
A(X_1 - \om^{2j}Y_1)- \om^jX_1Y_1(\om^j X_1^2 -\om^{2j} Y_1^2),
\end{gathered}
\end{equation}
 respectively, and are in ratio $(1  : - w^{2j})$. In other words,
 \eqref{sum1} fails precisely when the sum would be one of the points
 at infinity. Accordingly, we add them to the definition of $C$ and
 write 
\begin{equation}
\begin{gathered}
(X_1,Y_1) + (Y_1,X_1) + (1:-1:0) = 0, \\
(X_1,Y_1) + (\om Y_1,\om^2 X_1) + (1:-\om^2:0) = 0,
\\ (X_1,Y_1) + (\om^2  Y_1,\om X_1) +(1: -\om:0) = 0.
\end{gathered}
\end{equation}
 Note that, for example, $(X_1:Y_1:1), (Y_1:X_1:1), (1:-1:0)$ all
lie on the projective line $(X_1+Y_1)u_3 = u_1 + u_2$.

In view of \eqref{sumsym}, we see that $(1:-1:0)$ is the additive
identity and  the point $h_0 := (1: -\om:0)$ has order 3. Further, we see
that
\begin{equation}\label{E:infi}
(\om X_1, \om^2 Y_1) = (X_1,Y_1) + h_0, \qquad 
(\om^2 X_1, \om Y_1) = (X_1,Y_1) + 2h_0.
\end{equation}

We still need to define addition when $(X_1,Y_1) = (X_2,Y_2)$. In this
case, we construct the equivalent to the tangent line to the point at
$(X_1,Y_1)$ to make a double point and decree that the third
intersection point will be $-2(X_1,Y_1)$. Formal implicit
differentiation says that the ``slope'' to the curve $X^3 + Y^3 = A$
at $(X_1,Y_1)$ equals $-X_1^2/Y_1^2$, so we seek $\la 
\neq 0$ so that 
\begin{equation}\label{E:vimodern}
\begin{gathered}
(X_1 + \la )^3 + \left(Y_1 - \la \frac{X_1^2}{Y_1^2}\right)^3 
- A = 0. 
\end{gathered}
\end{equation}
Since $X_1^3 + Y_1^3 = A$, 
the left-hand side of \eqref{E:vimodern} is divisible by $\la^2$, so
\begin{equation}\label{E:double}
\begin{gathered}
 \la =
\frac{3X_1Y_1^3}{X_1^3-Y_1^3} \implies -2(X_1,Y_1) =\left( \frac{X_1(A +
  Y_1^3)}{X_1^3-Y_1^3}, \frac{-Y_1(A + X_1^3)}{X_1^3-Y_1^3} \right).
\end{gathered}
\end{equation}

We now set $A = x^3 + y^3 \in \mathbb C[x,y]$ and summarize the foregoing
discussion of  addition. Addition involving points at infinity is
specified by \eqref{sumsym} and \eqref{E:infi}. Otherwise,  
\begin{equation}\label{E:all}
\begin{gathered}
\text{If }  (X_1,Y_1) \neq (X_2,Y_2), \text{ then } 
(X_1,Y_1) + (X_2,Y_2) = (Z,W), \text{ where} \\
Z = \frac{(x^3+y^3)(Y_1-Y_2) +
  X_1X_2(X_1Y_2-X_2Y_1)}{(X_1^2X_2 + 
  Y_1^2Y_2)- (X_1X_2^2 + Y_1Y_2^2)}, \\ W = \frac{(x^3+y^3)(X_1-X_2) +
  Y_1Y_2(X_2Y_1-X_1Y_2)}{(X_1^2X_2 + 
  Y_1^2Y_2)- (X_1X_2^2 + Y_1Y_2^2)}; \\\text{Otherwise, }  2(X_1,Y_1)  =\left(
  \frac{-Y_1(2X_1^3+Y_1^3)}{X_1^3-Y_1^3} , \frac{X_1(X_1^3+  
 2Y_1^3)}{X_1^3-Y_1^3} \right).
\end{gathered}
\end{equation}

Note that $2(X_1,Y_1)= (X_1, Y_1) \circ (-v_4)$, 
and that in the extract cited in the Introduction, Vi\`ete in
\cite{V}, in effect, chooses a  slope for the line 
 to ensure that the cubic equation for $\la$ would have a double root
 at $\la = 0$, rendering its third root easy to find. (This tangent
 line was computed  almost 100 years before calculus was invented!)
Silverman  \cite[p.335]{S} explicitly derived \eqref{E:double}, with
regards to elliptic curves of the form $x^3+y^3 = A$ with $A \in \cc$,
although he did not make the reference to Vi\`ete.  

Finally, since $\om^2 = \la  + (1-\la) \om$ for $\la = -\om$, we observe that 
\begin{equation}\label{E:lines}
\begin{gathered}
(X_1,Y_1) + (X_1, \om Y_1) + (X_1, \om^2 Y_1)  = 0 \\
(X_1,Y_1) + (\om X_1, Y_1) + (\om^2 X_1, Y_1) = 0 \\
(X_1,Y_1) +  (\om X_1, \om Y_1)+ (\om^2 X_1, \om^2 Y_1) = 0.
\end{gathered}
\end{equation}

We now specialize this discussion to $\mathcal V$. First, we write the 
 affiliates of $(f,g) \in
\mathcal V$ in arrays to clarify these sums to zero over lines.
Write  $(f,g) = e_1$ and $(\om f, \om g) = e_2$ for short. Then
all affiliates can be expressed in terms of $e_1, e_2$ and $h_0$:
\small
\begin{equation}
\label{affiliatelist}
\begin{aligned}
(f, \om^2 g) &= e_{2} + 2h_{0} &
(\om f, \om^2 g) &= e_{1} + h_{0} &
(\om^2 f, \om^2 g) &= -e_{1} - e_{2}\\
(f, \om g) &= -e_{1} - e_{2} + h_{0} &
(\om f, \om g) &= e_{2} &
(\om^2 f, \om g) &= e_{1} + 2 h_{0}\\
(f,g) &= e_{1} &
(\om f, g) &= -e_{1} - e_{2} +2 h_{0} &
(\om^2 f, g) &= e_{2} + h_{0} \\ \\ 
(g, \om^2 f) &= -e_{2} + 2h_{0} &
(\om g, \om^2 f) &= -e_{1} + h_{0} &
(\om^2 g, \om^2 f) &= e_{1} + e_{2}\\
(g, \om f) &= e_{1} + e_{2} + h_{0} &
(\om g, \om f) &= -e_{2} &
(\om^2 g, \om f) &= -e_{1} + 2h_{0}\\
(g,f) &= -e_{1} &
(\om g, f) &= e_{1} + e_{2} + 2h_{0} &
(\om^2 g, f) &= -e_{2} + h_{0}.
\end{aligned}
\end{equation}
\normalsize

We now recall $h_0$ and identify two special points on $\mathcal V$:
\begin{equation}\label{E:players}
h_0 = (1: -\om:0), \qquad h_1 = (x,y), \qquad h_2 = (\om x, \om y),
\end{equation}
and let
\begin{equation}
{\mathcal V_0} = 
\{ m h_1 + n h_2 + t h_0\ :\  m, n \in \zz, \ t \in \{0,1,2\} \},
\end{equation}
where $m h_1 + n h_2 + t h_0$ is the {\it  canonical} expression
for $(f,g) \in \mathcal V_0$. (We henceforth reserve $m,n,t$ to the
description above.) We also recall the definition of 
an important subset of $\mathcal V_0$:
\begin{equation}
{\mathcal V_{1}} = 
\{ m h_1 + n h_2\ :\  m, n \in \zz \}.
\end{equation}
For $(0,0) \neq (m,n) \in \zz^2$, let 
\begin{equation}
T(m,n) = \{m h_1 + n h_2, m h_1 + nh_2+ h_0,  m h_1 + nh_2+ 2h_0\}
\end{equation}
denote the {\it $(m,n)$-trio}. 

We now begin to describe the ring-isomorphism between $\mathcal V_1$
and $\zz[\om]$ by analyzing  $v \circ w$. Our first results apply to
 $V_0$ as well. 
\begin{lemma}\label{simplecirc}
If $v = (f,g) \in \mathcal V$, then  
\begin{equation}\label{basiccirc}
h_1 \circ v = v \circ h_1 = v,\quad h_2 \circ v = v \circ h_2 =  \om
v,\quad h_2 \circ h_2 = \om h_2 = \om^2 h_1 = -h_1-h_2.
\end{equation}
\end{lemma}
\begin{proof}
The first identity is immediate from the definition of
composition. For the second one, note that
  $f$ and $g$ are homogeneous of degree 1, hence $f(\om x, \om y) =
  \om f(x,y)$ and $g(\om x, \om y) = \om g(x,y)$, so 
\begin{equation}
\begin{gathered}
v \circ h_2 =  (f(\om x, \om y),g(\om x, \om y)) = (\om f(x,y),\om g(x,y))
= h_2 \circ v.
\end{gathered}
\end{equation}
The final equation follows from the second and \eqref{E:lines}.
\end{proof}

We now make a simple, but consequential observation about left-distributivity.
\begin{lemma}\label{distrib}
If $v, v', w \in \mathcal V$ and $\tilde v = v + v'$, then $\tilde v
\circ w = v \circ w + v' \circ w$. Thus, $(m v + n v') \circ w = m v \circ w +
n v' \circ w$.  
\end{lemma}
\begin{proof}
Suppose $w = (f(x,y),g(x,y))$. Composition with $w$ amounts to 
the formal substitution $(x,y) \to (f(x,y),g(x,y))$; \eqref{E:circ} 
shows that substitution this is preserved by  the
varying definitions of addition, establishing the first assertion. The second
assertion follows from the first by induction. 
\end{proof}

\begin{theorem}\label{rotate}
If $v = (f,g) = mh_1 + nh_2 + th_0 \in \mathcal V_0$,  then 
\begin{equation}\label{E:rot}
(\om f, \om g) = -n h_1 +(m-n) h_2 + t h_0,\quad  (\om^2 f, \om^2 g) = (n-m)
h_1 - m h_2 + t h_0. 
\end{equation}
\end{theorem}
\begin{proof}
Lemmas~\ref{simplecirc} and \ref{distrib} imply that 
\begin{equation}
(\om f, \om g) = \om v = v \circ h_2 = m h_1 \circ h_2 + n
h_2 \circ h_2 + t h_0 \circ h_2 = m h_2 + n(-h_1-h_2)+th_0.
\end{equation}
The other equation follows from \eqref{E:lines}.
\end{proof}

For $x =  m h_1 + n h_2 \in \mathcal V_1$, define 
\begin{equation}\label{E:enchilada}
R(x) = R(m h_1 + n h_2) = m + n \om
\end{equation} 
Note that if $v = mh_1 + nh_2 \in \mathcal V_1$, then
Theorem~\ref{rotate} implies that 
\begin{equation}\label{E:rot2}
R(\om v) = -n + (m -n)\om = \om(m + n\om) = \om R(v).
\end{equation}
Using \eqref{affiliatelist} and Theorem~\ref{rotate}, once
we know the canonical expression for  
$(f,g)$, we know the canonical expressions for all of its 
affiliates. The canonical expressions for the solutions listed in the 
introduction are
\begin{equation}
\begin{gathered}
  v_{1} = h_{1}, \quad
  v_{3} = h_{1} + 2h_{2}, \quad
  v_{4} = -2h_{1}, \\
  v_{7} = -2h_{1} - 3h_{2}, \quad 
  v_{9} = -3h_{1},  \quad v_{12} = -2h_{1} - 4h_{2}.
\end{gathered}
\end{equation}
Note that $R(v_1)=1$, $R(v_3)=1+2\om = \om-\om^2 = i\sqrt{3}$
and $R(v_4) = -2$.

It is clear from \eqref{affiliatelist} and
Theorem~\ref{rotate} that each set of 18 affiliates is a
union of the six trios: 
\begin{equation}\label{E:trios}
\begin{gathered}
T(m,n),\quad T(-m,-n),\quad T(-n,m-n), \\  T(n,n-m),\quad
T(n-m,-m),\quad T(m-n,m), 
\end{gathered}
\end{equation}
and every $v \in \mathcal V_0$ is in a trio with some $w \in
\mathcal V_{1}$. As long as $(m,n) \neq (0,0)$, the six trios in
\eqref{E:trios} are distinct.

The argument of Theorem~\ref{rotate} extends to give a closed form for
composition in $\mathcal V$.
\begin{theorem}\label{circdecomp}
\label{compthm}
If $v = m h_1 + n h_2 + t h_0$ and $v' = m' h_1 + n' h_2 + t' h_0$, then
\begin{equation}\label{E:circformula}
\begin{gathered}
v \circ v' = m(m'h_1 + n' h_2 + t'h_0) + n(-n' h_1 + (m'-n') h_2 + t'
h_0)+ t h_0 \\ = (mm' - n n')h_1 + (mn'+m'n-nn')h_2 + ((m+n)t' + t)h_0 .
\end{gathered}
\end{equation}
\end{theorem}
\begin{proof}
Since $v\circ v' = m h_1 \circ v' + n h_2 + th_0
\circ v'$,  three applications of Lemma~\ref{distrib}, keeping
\eqref{basiccirc} in mind, give the result. 
\end{proof}

We note a crucial implication of Theorem~\ref{circdecomp} for
elements of $\mathcal V_1$  ($t=t'=0$):
\begin{equation}\label{E:gold}
R(v \circ v') = (mm' - nn') + \om(m n' + m'n - nn') =  R(v)R(w).
\end{equation}
For $y \in \zz[\om]$, let $N(y)$ denote the usual norm. We have 
\begin{equation}\label{E:norm}
\Phi(m,n) := N( m + n \om) = |m + n \om|^2 = ( m + n \om)( m + n \om^2) = m^2
- mn + n^2.  
\end{equation}
Since $N(y) = N(\pm \om^j y)$, we have
\begin{equation}\label{E:phi}
\begin{gathered}
\Phi(m,n)=\Phi(-m,-n)=\Phi(-n,m-n)=\Phi(n,n-m)\\=\Phi(n-m,m)=
\Phi(m-n,-m).
\end{gathered}
\end{equation}
 Further, \eqref{E:gold} implies that
\begin{equation}\label{E:Phiprop}
\Phi(mm'-nn',mn'+m'n-nn') = \Phi(m,n)\Phi(m',n');
\end{equation}
of course, \eqref{E:phi} and \eqref{E:Phiprop}
can also be verified directly. 
It will follow from  Theorem~\ref{mainthm} that $\mathcal V_{0} = 
\mathcal{V}$ and 
\begin{equation}\label{degrees}
d(m h_1 + n h_2 + t h_0) = \Phi(m,n) := m^2 - m 
n + n^2.
\end{equation}

\begin{corollary}\label{comm}
If $v, v' \in \mathcal V$ are given as above, then 
\begin{equation}
v \circ v' - v' \circ v =  (((m+n)t' + t) - ((m'+n')t + t'))h_0
\end{equation}
so $v \circ v'$ and $v' \circ v$ are in the
same trio. Furthermore, $v \circ v' = v' \circ v$  if and only if
$(m+n-1)t' \equiv 
(m'+n'-1)t \pmod {3}$, in particular, if $v, v' \in \mathcal V_1$.
\end{corollary}

We also remark  that Theorem~\ref{compthm} and the to-be-proved formula
\eqref{degrees} combine to 
imply that $d(v \circ v') = d(v)d(v')$, so that no cancellation
occurs in the composition. We note one more corollary to
\ref{compthm}, which follows from the bi-homogeneity in the pairs of
variables $(m,n)$ and $(m',n')$ of \eqref{E:circformula} for elements
of $\mathcal V_1$; the corollary is not generally true in $\mathcal V$.
\begin{corollary}\label{rv}
For $v, v' \in \mathcal V_1$ and $r \in \zz$, $(rv) \circ v' = v \circ (rv')
= r(v\circ v')$. In particular, taking $v' = h_1$, we have $rv = v
\circ (rh_1)$.
\end{corollary} 

\begin{theorem}\label{E:Ringo}
The map $R$ is a ring isomorphism between $\mathcal V_{1}$ (with the
operations of point-addition and composition) and
$\zz[\om]$. Furthermore, with the appropriate definitions of
multiplication by $\om$ and the usual complex conjugation,
\begin{equation}\label{E:extras}
R(\om v) = \om R(v), \qquad R(\overline v) = \overline{R(v)}.
\end{equation}
\end{theorem}
\begin{proof}
That $\mathcal{V}_{1}$ is a ring with respect to addition and
composition follows from Lemma~\ref{distrib} in one direction
(and from Corollary~\ref{comm} in the other). We remark that
Corollary~\ref{comm} implies that the full set $\mathcal V$ itself is
not a ring with 
composition as ``multiplication'', because the right-distributive law
fails. In particular,  if $v, v'\in \mathcal V$ then by
\eqref{E:circforms}, $h_0 = h_0 \circ v = h_0 \circ v' = h_0 \circ (v +
v')$, but $h_0 \neq 2h_0$.

Clearly, $R$ is a bijection and $R(v+w) = R(v)+R(w)$. If $v = mh_1 +
nh_2$ and $w = m'h_1 + n'h_2$, then as we have seen in \eqref{E:gold},
$R(v\circ w) = R(v)R(w)$. That $R(\om v) = \om R(v)$ was
shown in equation \eqref{E:rot2}. For the second statement, we needn't
concern ourselves with points at infinity, and the exact formulas of
\eqref{E:all} imply that complex conjugation factors through addition, so that
\begin{equation}
(X_1,Y_1) + (X_2,Y_2) = (Z,W) \implies \overline{(X_1,Y_1)} +
\overline{(X_2,Y_2)} = \overline{(Z,W)}.
\end{equation}
Since $\overline{h_1} = h_1$ and  $\overline{h_2} = \overline{(\om x, \om
  y)} = (\om^2 x, \om^2 y) = \om h_2 = -h_1-h_2$, we see that if $v =
mh_1 + nh_2$, then 
\begin{equation}\label{E:Rconjugate}\begin{gathered}
\overline{v} = m h_1 + n (-h_1-h_2) \implies \\ R(\overline{v}) = m
+n(-1-\om) = m + n \om^2 = \overline{m + n\om} = \overline{R(v)}. 
\end{gathered}
\end{equation}
\end{proof}

\begin{corollary}\label{inv}
If $v = m_0 h_1 + n_0 h_2 \in \mathcal V_1$
 and $w = m_1 h_1 + n_1 h_2$, then there exists
 $v' \in \mathcal V$ such that $v = v' \circ w$ if and only if
$\frac {m_0+n_0\om}{m_1+n_1\om} \in \zz[\om]$; that is, if and only if
\begin{equation}\label{E:div}
m_0m_1 + n_0 n_1 \equiv m_0n_1 \equiv m_1 n_0 \pmod {m_1^2 - m_1n_1 + n_1^2}
\end{equation}
\end{corollary}
\begin{proof}
A routine calculation shows that
\begin{equation}
\begin{gathered}
\frac {m_0+n_0\om}{m_1+n_1\om} 
= \frac{ m_0m_1 + n_0n_1 - m_0n_1 + (m_1n_0 - m_0n_1)\om}{m_1^2-m_1n_1+n_1^2}.
\end{gathered}
\end{equation}

\end{proof}

\begin{corollary}\label{realmaker}
If $v = m h_1 + n h_2 \in \mathcal V_1$, then
\begin{equation}\label{E:real}
v \circ \bar v = N(R(v)) h_1 = (m^2-mn+n^2) h_1.
\end{equation}
\end{corollary}
\begin{proof}
It follows from \eqref{E:Rconjugate} that $R(v)R(\bar v) = N(R(v))$.
\end{proof}

If particular, since $g_3 = \bar f_3$, $\bar v_3 = - v_3$ and we 
recover that $v_3 \circ v_3= v_9$. It  follows from \eqref{E:Rconjugate} that 
$\overline{mh_1} = mh_1$ for all $m \in \zz$; thus
Corollary~\ref{realmaker} implies that each $v \in \mathcal V_1$ has a
``composition multiple'' which is real. Observe that $m h_1 + n h_2$
and $n h_1 + m h_2$ are not, in general, affiliates, although
$\Phi(m,n) = \Phi(n,m)$. 

\begin{corollary}\label{conju}
If $v = m h_1 + n h_2$, then $\om \overline{v} = n h_1 + m h_2$.
\end{corollary}
\begin{proof}
This follows immediately from $n+ m\om = \om(\overline{m + n\om})$.
\end{proof}

\begin{corollary}\label{oddeven}
If $v \in \mathcal V_1$, then $v$ and $\bar v$ are affiliates if and
only if $v$ is an affiliate of  $m v_1$ or $m v_3$ for some integer
$m$. 
\end{corollary}
\begin{proof}
This follows from a somewhat tedious comparison of \eqref{E:trios} for
$(m,n)$ and $(n,m)$. Equality holds if $mn(m-n)=0$ or
$(m+n)(m-2n)(2m-n)=0$, which give multiples of $v_1$ and $v_3$
respectively. 
\end{proof}
We note that $\overline{mv_1} = mv_1$ and $\overline{mv_3} = -mv_3$.
It follows that the number of solutions of degree $d$, $f(d)$, is even
unless $d = m^2$ or $d = 3m^2$. 

We now turn to proving Theorem~\ref{conseqthm} for points in
$\mathcal{V}_{0}$, assuming Theorem~\ref{mainthm}. We show that
assertions (1) through (6) hold for $v, w\in \mathcal V_{1}$;
since the components of $v + th_0$ differ from $v$ by powers of $\om$,
which do not affect the assertions,
the claimed results will also hold for $\mathcal V_0$.  

\begin{proof}[Proof of Theorem~\ref{conseqthm}]
We start with $(1)$. Let $v =  m h_1 + n h_2$ and $v' = (y,x)= -
h_1$. Then $v \circ v' = (f(y,x),g(y,x))$ by \eqref{E:circ}. On the other
hand, $R(v \circ v') = R(v)R(v') = -R(v) = -m -n\om = R(-v)$, hence
$v \circ v' = -v$; that is,  $(f(y,x),g(y,x)) = (g(x,y),f(x,y))$.

Item $(2)$ is Corollary~\ref{comm}.

Item $(3)$ follows from Theorem~\ref{mainthm} together with the
observation (from \eqref{E:all}) that if $v$ and $v'$ have
coefficients in $\qq(\omega)$, then so does $v + v'$.

To prove $(4)$, let $v = m h_{1} + n h_{2}$ and note that 
$d(v) = (m+n)^2-3mn \equiv 0 \pmod{3}$ implies
that $m+n \equiv 0 \pmod{3}$. Applying \eqref{E:div} from
Corollary~\ref{inv} gives $-m-2n\equiv -2m \equiv -n \pmod{3}$ as a
condition for the existence of $v' \in \mathcal V_1$ so that $v = v'
\circ v_3$, and these are satisfied when $3 \ | \ m+n$. 
 If $v' = (\frac {p'}{r'}, \frac
{q'}{r'})$, then by 
\eqref{E:circforms},
\begin{equation}
\begin{gathered}
p(x,y) = p'(\ze^{-1} x^3 + \ze y^3,\ze x^3 + \ze^{-1} y^3), \\
q(x,y) = q'(\ze^{-1} x^3 + \ze y^3,\ze x^3 + \ze^{-1} y^3), \\
r(x,y) = \sqrt{3}\ xy\  r'(\ze^{-1} x^3 + \ze y^3,\ze x^3 + \ze^{-1} y^3),
\end{gathered}
\end{equation}
which verifies the asserted shape. (Compare with the earlier discussion
of \eqref{E:lucas}.)

Next, we prove $(5)$. Since $(m+n)^2 \equiv d(v) \equiv 1
\pmod{3}$, for one choice of sign (say $+$), we have
 $\pm v = m h_{1} + n h_{2}$, where $m+n \equiv 1 \pmod{3}$. (The
 choice of sign amounts to a possible permutation of $f$ and $g$.) Let
$v = (f,g) = (p/r,q/r)$. Since $(\om x,y) = -h_1-h_2+2h_0$,
Theorem~\ref{compthm} now implies that 
\[
  v \circ (\omega x, y) = (n-m) h_{1} - m h_{2} + 2(m+n)h_{0} =  (n-m)
  h_{1} - m h_{2} + 2h_{0},  
\]
which by \eqref{affiliatelist} and
Theorem~\ref{rotate} 
equals $(\om f,  g)$. In other words, 
\begin{equation}\label{E:1.25}
\left( \frac{p(\om x, y)}{r(\om x,y)}, \frac{q(\om x, y)}{r(\om x,y)}
\right) = \left( \om \frac{p( x, y)}{r(x,y)}, \frac{q(x, y)}{r(x,y)}
\right).
\end{equation}
Thus,  $p(\om x, y) r(x,y) = \om p(x,y) r(\om x, y)$ and  $q(\om
x, y) r(x,y) =  q(x,y) r(\om x, y)$. Since $p(x,y)$ and $r(x,y)$
are relatively prime, we have $p(x,y) | p(\om x, y)$; since they have
the same degree, it follows that $p(\om x,y) = c_p p(x,y)$ for $c_p
\in \mathbb C$.  Similarly, $q(\om x,y) = c_q q(x,y)$ and $r(\om x,y) = c_r
r(x,y)$. Since $p,q,r \neq 0$, examination at any non-zero monomial shows that
each constant is a power of $\om$ and so all powers of $x$ occuring in
$p$ with non-zero coefficient are congruent modulo 3, and similarly
for $q$ and $r$. Since $d \equiv 1 \pmod{3}$, the choices are $p(x,y)
= xP(x^3,y^3), 
yP(x^3,y^3)$ or $x^2y^2 P(x^3,y^3)$ for some polynomial $P$; 
 $q(x,y) = xQ(x^3,y^3), yQ(x^3,y^3)$ or $x^2y^2 Q(x^3,y^3)$ for some
 polynomial $Q$; and  $r(x,y) = R(x^3,y^3), xy^2R(x^3,y^3)$ or $x^2y
 R(x^3,y^3)$ for some polynomial $R$. Since $p$ 
 and $q$ are relatively prime, there cannot be a common factor
 of $x$ or $y$, hence $(p(x,y),q(x,y))$ is either $(xP(x^3,y^3),
 yQ(x^3,y^3))$ or $(yP(x^3,y^3),xQ(x^3,y^3))$. 
 Upon dividing the components of either side of \eqref{E:1.25}, we find that
 \begin{equation}
 \frac{p(\om x, y)}{q(\om x,y)} = \om  \frac{p(x, y)}{q(x,y)}, 
 \end{equation}
 hence $p(x,y) = xP(x^3,y^3)$ and $q(x,y) = yQ(x^3,y^3)$;
 \eqref{E:1.25} now implies that  $r(x,y) = r(\om x, y)$, so $r(x,y) =
 R(x^3,y^3)$. 

Item $(6)$ follows immediately from $(4)$ and $(5)$.

To prove $(7)$, note that \eqref{E:infi} $\overline{h_0} =
\overline{(1:-\om:0)} = (1:-\om^2:0) = 2h_0$ and so \eqref{E:infi} and
Theorem~\ref{E:Ringo} imply that 
\begin{equation}\label{E:conjugate}
\overline{ m h_1 + n h_2 + t h_0} = m h_1 + n (-h_1-h_2) + t(2h_0).
\end{equation}
If $v$ is real, then $v = \overline v$ so that $nh_1 + 2nh_2 - th_0 =
0$, hence $n=t=0$. Note also that if $v =(f,g) = rh_1$, then $f,g \in
\qq(x,y)$. 

Finally, we turn to $(8)$.  Since each solution has 18 affiliates and
Theorem~\ref{mainthm} implies that 
$d(mh_{1} + nh_{2} + th_{0}) = m^{2} - mn + n^{2}$, $t \in \{0,1,2\}$, we have
 \begin{equation}\label{E:genfn}
1 + 6\sum_{d=1}^\infty f(d)z^d =  \sum_{m=-\infty}^{\infty}
\sum_{n=-\infty}^{\infty} z^{m^2-mn+n^2}.
\end{equation}
It is fairly well-known that
\begin{equation}\label{E:genfn2}
 \sum_{m=-\infty}^{\infty}
\sum_{n=-\infty}^{\infty} z^{m^2-mn+n^2} =  1 + 6 \sum_{i=0}^\infty \left(
  \frac{z^{3i+1}}{1-z^{3i+1}} - \frac{z^{3i+2}}{1-z^{3i+2}}\right).
\end{equation}
The equations \eqref{E:genfn} and \eqref{E:genfn2} combine to imply
\eqref{E:refsugg}.
The identity \eqref{E:genfn2} has a convoluted
history, as described by our colleague Bruce Berndt in, for example,
\cite[p.78]{B} and \cite[pp.196-199]{BR}, and by Hirschhorn in
\cite{H}.  Its arithmetical equivalent is a special
case of an 1840 theorem of Dirichlet. It was found independently by
Lorenz and Ramanujan. 

It follows from \eqref{E:refsugg} that $f(p^k) = k+1$ if
$p \equiv 1 \pmod{3}$; if $p \equiv 2
\pmod{3}$, then $f(p^k)$ equals 0 or 1, depending on whether $k$ is
odd or even. Since  $f(d)$ is multiplicative, $f(n) > 0$ implies 
that no prime $\equiv 2 \pmod{3}$ can appear to an odd power in the
prime factorization of $n$.  We note also that 
$\{f(n)\}$ is unbounded as $n \to \infty$. Since the taxicab number 
$1729 = 7 \cdot 13 \cdot 19$, we have $f(1729) = 8$: there are 8 solutions to
\eqref{titleequation} in which $p$ and $q$ have degree 1729.  
\end{proof}

\section{Proof of Theorem~\ref{mainthm}}
\label{mwtheory}

In this section, we will prove Theorem~\ref{mainthm}. We will 
consider solutions to
\[
  E : p^{3} + q^{3} = (x^{3} + y^{3}) r^{3}
\]
where $0 \neq p, q, r \in \cc[x,y]$ are homogeneous polynomials
with $\deg(p) = \deg(q) = \deg(r) - 1$. 

For such a point $P = (a(x,y) : b(x,y) : c(x,y))$, the map
\[
  \phi_{P}(p : q : r) = (a(p,q) : b(p,q) : c(p,q) r)
\]
is a morphism from $E$ to itself. There are two basic properties 
that immediately follow from the definition. First,
\begin{equation}
\label{identity}
  \phi_{P}(x : y : 1) = (a(x,y) : b(x,y) : c(x,y)) = P.
\end{equation}
Second, for any point $P$, the map $\phi_{P}$ permutes the
points at infinity: $(1 : -1 : 0)$, $(1 : -\omega : 0)$ and 
$(1 : -\omega^{2} : 0)$.

Before we continue, we need a lemma.
\begin{lemma}
\label{inforder}
The point $P = (x : y : 1) \in \mathcal{V}$ has infinite order.
\end{lemma}
\begin{proof}
Define the homomorphism $\phi : E \to E'$ by setting $x = 2$ and
$y = 1$. Thus
\[
  E' : p^{3} + q^{3} = 9 r^{3}.
\]
We have $\phi((x : y : 1)) = (2 : 1 : 1)$. Using standard 
techniques (e.g. Prop. VII.3.1(b) or Cor. VIII.7.2 in \cite{Silverman}),
 one can compute that $E'(\qq)$, the group of rational 
points on $E'$, is isomorphic to $\zz$, and is generated by $(2 : 
1 : 1)$. It follows that $P$ has infinite order, since a 
homomorphic image of $P$ also has infinite order.
\end{proof}

Let $\mathcal{V}_{\infty} = \{ P \in \mathcal{V} : \phi_{P}(0) = 
0 \}$ be the subgroup of points $P \in \mathcal{V}$ so that 
$\phi_{P}$ fixes the chosen point at infinity. Recall that any 
polynomial map $\phi : E \to E$ with $\phi(0) = 0$ is called an 
{\it isogeny}. Theorem III.4.8 of \cite{Silverman} implies that if 
$\phi$ is an isogeny, then $\phi(P + Q) = \phi(P) + \phi(Q)$.  
The set of all isogenies from $E$ to itself is denoted $\End(E)$ 
and is called the endomorphism ring of $E$. The two ring 
operations are addition (in the group law, so $(\phi_{1} + 
\phi_{2})(R) = \phi_{1}(R) + \phi_{2}(R)$), and function 
composition. Our approach to proving Theorem~\ref{mainthm} will 
be to define a ring structure on $\mathcal{V}_{\infty}$, and 
prove that $\mathcal{V}_{\infty} \cong \End(E)$, and
finally show that $\mathcal{V}_{\infty} = \mathcal{V}_{1} = \{
m h_{1} + n h_{2} : m, n \in \zz \}$.

\begin{lemma}
\label{sum}
For any two points $P, Q \in \mathcal{V}$, we have
\[
  \phi_{P + Q} = \phi_{P} + \phi_{Q}.
\]
\end{lemma}
\begin{proof}
From \eqref{identity}, we have
\begin{align*}
  \phi_{P+Q}(x : y : 1) &= P + Q\\
  &= \phi_{P}(x : y : 1) + \phi_{Q}(x : y : 1).
\end{align*}
Thus, the point $(x : y : 1)$ is sent to $0$ under the map 
$\phi_{P + Q} - \phi_{P} - \phi_{Q}$. If $S = \phi_{P+Q}(0) - 
\phi_{P}(0) - \phi_{Q}(0)$, then $F = \phi_{P+Q} - \phi_{P} - 
\phi_{Q} - S$ is a morphism from $E$ to itself that fixes $0$. 
Thus, $F$ is an isogeny. Any morphism between two curves is 
either constant, or each point has finitely many preimages. It 
follows that $\ker F$ is either finite, or all of $E$. Since $F$ 
is an isogeny,
\[
  F([3n] (x : y : 1)) = [3n] F((x : y : 1)) = [3n](-S) = [n] ([3] (-S)) = 0. 
\]
Here, and in the rest of the section, $[m](p:q:r)$ is used instead of
$m(p:q:r)$ for clarity. 
By Lemma~\ref{inforder}, $(x : y : 1)$ has infinite order,
and hence the kernel of $F$ is infinite. This implies
that $F$ is the zero map, and so
\[
  \phi_{P+Q}(R) - \phi_{P}(R) - \phi_{Q}(R) = S.
\]
for any $R$. Setting $R = (x : y : 1)$ we see that $S = 0$, and
$\phi_{P+Q} = \phi_{P} + \phi_{Q}$.
\end{proof}

Recall that $h_{0} = (1 : -\omega : 0) \in \mathcal{V}$ and 
$2h_{0} = (1 : -\omega^{2} : 0)$. Clearly
\[
  \phi_{h_{0}}(R) = h_0, \qquad \phi_{2h_{0}}(R) = 2h_0
\]
for all $R \in \mathcal{V}$. It follows that for any point $P \in 
\mathcal{V}$, either $P$, $P - h_{0}$ or $P - 2h_{0} \in 
\mathcal{V}_{\infty}$. Hence,
\[
  \mathcal{V} \cong \mathcal{V}_{\infty} \times \langle h_{0} \rangle.
\]

\begin{lemma}
\label{ringstruct}
The subgroup $\mathcal{V}_{\infty} \subseteq \mathcal{V}$ can be given
the structure of a ring by defining $P \cdot Q = \phi_{P}(Q)$.
\end{lemma}
\begin{proof}
We know that $\mathcal{V}_{\infty}$ is an abelian group. We must show that the 
multiplication operator is associative and distributive. By \eqref{identity},
\[
  \phi_{P}(x : y : 1) = P,
\]
we have
\[
  P \cdot Q = \phi_{P}(Q) = \phi_{P}(\phi_{Q}(x : y : 1)).
\]
Since $\phi_{S}(x : y : 1) = S$ for any $S \in \mathcal{V}$, it follows that
$\phi_{P \cdot Q}(x : y : 1) = P \cdot Q = \phi_{P}(\phi_{Q}(x : y : 1))$.
Thus, $(x : y : 1)$ is in the kernel of the isogeny
$\phi_{P \cdot Q} - \phi_{P} \circ \phi_{Q}$. By Lemma~\ref{inforder},
the kernel is therefore infinite and hence $\phi_{P \cdot Q} = \phi_{P}
\circ \phi_{Q}$. The associativity then follows from the fact
that function composition is associative. To prove the distributive law,
we use that $\phi_{P}$ is an isogeny and hence
\begin{align*}
  P \cdot (Q + R) &= \phi_{P}(Q + R)\\
  &= \phi_{P}(Q) + \phi_{P}(R)\\
  &= (P \cdot Q) + (P \cdot R).
\end{align*}
Thus, $\mathcal{V}_{\infty}$ naturally has the structure of a ring.
\end{proof}

\begin{lemma}
\label{iso}
The map $\tau : \mathcal{V}_{\infty} \to \End(E)$ given by
\[
  \tau(P) = \phi_{P}
\]
is an isomorphism of rings. 
\end{lemma}
\begin{proof}
Lemma~\ref{sum} implies that $\tau(P + Q) = \tau(P) + \tau(Q)$. 
In the proof of Lemma~\ref{ringstruct}, we showed that $\tau(P 
\cdot Q) = \tau(P) \circ \tau(Q)$. Thus, $\tau$ is a ring 
homomorphism. If $\tau(P) = 0$, then $\phi_{P} = 0$ and so 
$\phi_{P}((x : y : 1)) = P = 0$. Hence, $\tau$ is injective.

Conversely, if $\phi \in \End(E)$, and $P = \phi(x : y : 1)$, then
$\phi - \phi_{P}$ has $(x : y : 1)$ in its kernel. Thus, the
kernel of $\phi - \phi_{P}$ is infinite and hence $\phi = \phi_{P}$.
It follows that $\phi = \tau(P)$ and so $\tau$ is surjective. 
\end{proof}

A similar argument identifying the Mordell-Weil group of an elliptic
surface with the endomorphism ring was given by Frank de Zeeuw
in his master's thesis \cite{Zeeuw}.

Now, we will prove our main result.
\begin{proof}[Proof of Theorem~\ref{mainthm}]
In light of the fact that
\[
  \mathcal{V} \cong \mathcal{V}_{\infty} \times \langle T \rangle,
\]
and that $\mathcal{V}_{\infty}$ is isomorphic to $\End(E)$ by Lemma~\ref{iso}, it
suffices to determine $\End(E)$. Theorem VI.6.1(b) of
\cite{Silverman} states that if $E$ is an elliptic curve defined over
a field of characteristic zero, then $\End(E)$ is isomorphic to either
$\zz$ or an order in an imaginary quadratic field.  Observe that
$\End(E)$ contains the map defined by
\[
  \phi((p : q : r)) = (\omega p : \omega q : r).
\]
Hereafter we will refer to the map $\phi$ as $[\omega]$.
This map fixes $(1 : -1 : 0)$, satisfies $[\omega]^{3} = 1$,
and sends $(x : y : 1)$ to $(\omega x : \omega y : 1)$. It follows
that $\End(E) \cong \zz[\omega]$ and
\[
  \mathcal{V}_{\infty} = \mathcal{V}_{1} = 
\langle (x : y : 1), (\omega x : \omega y : 1)\rangle
  \cong \zz \times \zz.
\]

Now, we will prove that $d(m h_{1} + n h_{2} + t h_{0}) = m^{2} - 
mn + n^{2}$. It suffices to prove this with $t = 0$, since $m 
h_{1} + n h_{2}$ is an affiliate of $m h_{1} + n h_{2} + t h_0$. 

If $P := m h_{1} + n h_{2} \in \mathcal{V}_{\infty}$, it is easy to see 
that the degree of $P$ is the same as the degree of the map 
$\phi_{P} : E \to E$. In this case,
\begin{align*}
  \phi_{P}((x : y : 1)) &= m h_{1} + n h_{2}\\
  &= m (x : y: 1) + n (\omega x : \omega y : 1)\\
  &= [m + n \omega] (x : y : 1).
\end{align*}
Thus, $\phi_{P} = [m + n \omega]$. The ring $\End(E)$ is endowed 
with an involution $\hat{\cdot}$ that satisfies
\begin{align*}
  \widehat{\lambda + \phi} &= \hat{\lambda} + \hat{\phi}\\
  \widehat{\lambda \circ \phi} &= \hat{\phi} \circ \hat{\lambda}\\
  \phi \circ \hat{\phi} &= [\deg\phi]
\end{align*}
(see Theorem III.6.2 of \cite{Silverman}). This, together with
the fact that $\deg([m]) = m^{2}$ implies that
\[
  \hat{[\omega]} = [\omega^{2}].
\]
This implies that
\[
  \widehat{[m + n \omega]} = [m + n \omega^{2}]
\]
and so
\begin{align*}
  [\deg([m + n \omega])] &= [m + n \omega] [m + n \omega^{2}]\\
  &= [m^{2} + (mn \omega + mn \omega^{2}) + n^{2}]\\
  &= [m^{2} - mn + n^{2}].
\end{align*}
Since the degree of $P$ equals $\deg \phi_{P} = \deg [m + n 
\omega]$, we have that the degree of $P$ is $m^{2} - mn + n^{2}$, 
as desired.
\end{proof}

\section{Related results and open questions}
\label{more}

We conclude with a brief discussion of some related
Diophantine equations. It is classically known that if $F(x,y)$ is a
binary cubic form, then after an invertible linear transformation in $(x,y)$,
$F(x,y)$ has one of the following three shapes: $x^3, x^3+y^3,
x^2y$. It is natural to wonder whether there are solutions to
\eqref{titleequation} in the other two cases.
\begin{theorem}
The equations
\begin{equation}\label{E:x3}
 p^3(x,y) + q^3(x,y) = x^3r^3(x,y),
\end{equation}
\begin{equation}\label{E:x2y}
 p^3(x,y) + q^3(x,y) = x^2y\ r^3(x,y),
\end{equation}
have no non-trivial solutions in forms $p, q, r \in \mathbb C[x,y]$.
\end{theorem}
\begin{proof}
Any solution to \eqref{E:x3} would be a solution to the Fermat
equation $X^n+Y^n=Z^n$ for $n=3$ over $\mathbb C[t]$, upon setting
$(x,y) = (1,t)$. The non-existence of such non-constant solutions was proved by
Liouville in 1879. (See the exposition in \cite[pp.263-265]{Rib}.)

Assume \eqref{E:x2y} has a solution and 
rewrite as 
\begin{equation}\label{E:factor}
x^2 y\ r^3 = (p+q)(p+\om q)(p+\om^2 q).
\end{equation} 
Let $\mathcal F = \{p+\om^j q: j=0,1,2\}$.
Note that $\mathcal F$ is linearly dependent: $\sum
\om^j(p+\om^jq) = 0$, hence any polynomial that divides two elements
of $\mathcal F$ divides
the third, and also divides $p$ and $q$.
Let $(p_0,q_0,r_0)$ be a solution of \eqref{E:factor} in which $d =
\deg r_0$ is minimal. If $d=0$, then $p_0$ and $q_0$ must be linear and the
product of the elements in $\mathcal F$ is $x^2y$, hence 
 $x$ must divide two of them, and so $x | p_0,q_0$, a contradiction.
Now suppose $d \ge 1$ and suppose $\pi$ is an
irreducible factor of $r_0$. If $\pi$ divides two elements of $\mathcal F$,
then, as before, $\pi$ divides $p,q$ and  $(p_0 : \pi,q_0/\pi : r_0/\pi)$ is a
solution to \eqref{E:x2y} of lower degree. It follows that if  $\pi^m$
is a factor of $r_0$, then $\pi^{3m}$ is concentrated in one member of
$\mathcal F$. We may thus write $r_0 = s_0s_1s_2$ so that $s_j^3 | p_0
+ \om^j q_0$. 
Since the degrees of $\{p_0 + \om^j q_0\}$ are equal, \eqref{E:factor}
implies that the three remaining factors, $\{x,x,y\}$, are either
dispersed, one to each $p_0 + \om^j q_0$, or combined in a single
factor. In the first case, we may again conclude that  $x | p_0, q_0$, and
\eqref{E:factor} implies that $x | r_0$, a contradiction. In the second
case, suppose without loss of generality that $x^2y | p_0 + q_0$. Then we
have $p_0+q_0 = x^2y\ \!  s_0^3$, $p_0 + \om q_0 = s_1^3$,
$p_0+\om^2q_0 = s_2^3$, and 
the linear dependence on the elements of $\mathcal F$ implies that
\begin{equation}
x^2ys_0^3 = -\om s_1^3 -\om^2 s_2^3,
\end{equation}
which, after the absorption of constants, is a solution to \eqref{E:x2y}.
If $\deg s_1 = d$, then $\deg s_2=d$, $\deg s_0 = d-1$ and $\deg r =
\deg s_0 + \deg s_1 + \deg s_2 = 3d-1 > d$, contradicting its supposed
minimality and completing the descent.
\end{proof}

We now show that Theorem 1.2(4,5) contains, in effect, the solution to two other
Diophantine equations. 
\begin{theorem}
Any solution in forms $a,b,c \in \cc[x,y]$ to either of the equations
\begin{equation}\label{E:xyx+y}
 a^3(x,y) + b^3(x,y) = xy(x+y)\ \! c^3(x,y),
\end{equation}
\begin{equation}\label{E:xy}
 x\ \! a^3(x,y) + y\ \! b^3(x,y) = (x+y)\ \!  c^3(x,y)
\end{equation}
can be directly derived from a solution to \eqref{titleequation}.
\end{theorem}
\begin{proof}
If \eqref{E:xyx+y}
holds, then by taking $(x,y) \mapsto (x^3,y^3)$, we see that
 \begin{equation}
a^3(x^3,y^3) + b^3(x^3,y^3) = x^3y^3(x^3+y^3)c^3(x^3,y^3),
\end{equation}
hence $(a(x^3,y^3) : b(x^3,y^3) : xy\ c(x^3,y^3)) \in \mathcal V$, and
$\deg(a(x^3,y^3)) = 3d'$. In the language of Theorem 1.2(4), we have
$(a,b,c) = (P,Q,R)$; again, compare with \eqref{E:lucas}.

Similarly, suppose \eqref{E:xy} holds; take $(x,y) \mapsto (x^3,y^3)$ to
obtain
\begin{equation}
 x^3\ \! a^3(x^3,y^3) +  y^3\ \! b^3(x^3,y^3) = (x^3+y^3)c^3(x^3,y^3).
\end{equation}   
Thus $(xa(x^3,y^3) : yb(x^3,y^3) : c(x^3y^3)) \in \mathcal V$ and 
$\deg (xa(x^3,y^3)) = 3d'+1$, so that in the language of Theorem 1.2(5),
we have $(a,b,c) = (P,Q,R)$. 
\end{proof}

The subject of equal sums of two cubes has a very long history. For
example, the {\it Euler-Binet} formulas (see e.g. \cite[\S 13.7]{HW})
give a complete parameterization to the equation
\begin{equation}
X^3 + Y^3 = U^3 + V^3
\end{equation}
over $\mathbb Q$, although an examination of the proof in \cite{HW} 
shows that it also
applies to any field $F$ of characteristic zero, such as $\mathbb
C(x,y)$. The parameterization is:
\begin{equation}
\begin{gathered}
X = \la(1 - (a-3b)(a^2+3b^2), \qquad Y = \la((a+3b)(a^2+3b^2)-1), \\
U = \la( (a+3b) - (a^2+3b^2)^2), \qquad V = \la((a^2+3b^2)^2 -(a-3b)),
\end{gathered}
\end{equation}
where $a, b, \la \in F$. One can easily solve for $(a,b,\la)$ for which
$X=x, Y=y, U=f, V=g$, although the derivation assumes that
$f^3+g^3=x^3+y^3$, so it is unhelpful in finding solutions to
\eqref{E:mainequation}. Further, these solutions do not necessarily come
from simple choices of $(a,b,\la)$. For example, in the case of
\eqref{viete}, a computation shows that $(X,Y,U,V) = (x,y,f_4,g_4)$
arises (uniquely) from
\begin{equation}
a = \frac{2x^2+5xy+2y^2}{2(x^2+xy+y^2)},\quad
 b = -\frac{3xy(x+y)}{2(x^3-y^3)},\quad 
\la = - \frac{(x-y)^3}{9xy}. 
\end{equation}
If $(g_4,f_4)$ is taken instead of $(f_4,g_4)$, then $a$ is a
quotient of two quartics, $b$ is a quotient of two quintics and $\la$
is a quintic divided by a quartic.

Finally, we look at some more general sums of two cubes. If $h,k, F
\in \mathbb C(x,y)$, $h^3 + k^3 = F$, $w = (h,k)$  and $v = (f,g) \in
\mathcal V$, 
then there is (at least) a one-sided composition on all solutions to $X^3+Y^3 =
F$, given by $v \circ w = (f(h,k),g(h,k))$. This follows from 
\begin{equation}\label{E:so2c}
f^3(h,k) + g^3(h,k) = h^3+k^3 = F.
\end{equation}
For example, with $F(x,y) = 2x^6 - 2y^6$ and $\ga = 2^{1/3}$,
\begin{equation}\label{E:last}
(x^2 + x y - y^2)^3 + (x^2 - x y - y^2)^3 = (\ga x^2)^3 + (- \ga y^2)^3
= 2x^6 - 2y^6.
\end{equation}
However, there is clearly no $v = (f,g) \in \mathcal V$ so that $x^2
+ xy - y^2 = f(\ga x^2,-\ga y^2)$. Moreover, there are other solutions to 
\begin{equation}\label{E:sextic}
a^3(x,y) + b^3(x,y) = (2x^6 - 2y^6)c^3(x,y).
\end{equation}  
For example, 
\begin{equation}
a_0(x,y) = x^3 + \tfrac i{\sqrt 3}y^3,\quad b_0(x,y) = x^3 -  \tfrac
i{\sqrt 3}y^3,\quad c_0(x,y) = x, 
\end{equation}
(and $(a_0(y,x),b_0(y,x),-c_0(y,x))$) do not arise from composition of
either solution of \eqref{E:last} with $\mathcal V$. We look forward
to finding the complete structure of the solutions to \eqref{E:sextic}.


\end{document}